\documentclass[a4paper,12pt, twoside, pdftex]{amsart}

\usepackage{fullpage}
\usepackage{amsmath, amsthm, amssymb}
\usepackage{txfonts}
\usepackage{amsfonts}
\usepackage{prettyref}
\usepackage{mathrsfs}
\usepackage[all]{xy}
\usepackage[utf8]{inputenc}

\newtheorem{dummy}{dummy}[section]
\newtheorem{lemma}[dummy]{Lemma}
\newtheorem{theorem}[dummy]{Theorem}

\newtheorem{corollary}[dummy]{Corollary}

\theoremstyle{definition}
\newtheorem{definition}[dummy]{Definition}
\newtheorem*{definition*}{Definition}

\newtheorem{remark}[dummy]{Remark}

\newtheorem*{acknowledgements}{Acknowledgements}

\numberwithin{equation}{section}

\newrefformat{th}{Theorem~\ref{#1}}
\newrefformat{cr}{Corollary~\ref{#1}}
\newrefformat{lm}{Lemma~\ref{#1}}
\newrefformat{dl}{Definition-Lemma~\ref{#1}}
\newrefformat{df}{Definition~\ref{#1}}
\newrefformat{cl}{Claim~\ref{#1}}
\newrefformat{sl}{Sublemma~\ref{#1}}
\newrefformat{pr}{Proposition~\ref{#1}}
\newrefformat{cj}{Conjecture~\ref{#1}}
\newrefformat{st}{Step~\ref{#1}}
\newrefformat{sc}{Section~\ref{#1}}
\newrefformat{df}{Definition~\ref{#1}}
\newrefformat{rm}{Remark~\ref{#1}}
\newrefformat{q}{Question~\ref{#1}}
\newrefformat{pb}{Problem~\ref{#1}}
\newrefformat{cd}{Condition~\ref{#1}}
\newrefformat{eg}{Example~\ref{#1}}
\newrefformat{he}{Heore~\ref{#1}}
\newrefformat{fg}{Figure~\ref{#1}}
\newrefformat{tb}{Table~\ref{#1}}
\newrefformat{as}{Assumption~\ref{#1}}

\usepackage{tikz}
\usetikzlibrary{backgrounds}

\newcommand{\pref}{\prettyref}

\newcommand{\Br}{\operatorname{Br}}
\newcommand{\coh}{\operatorname{coh}}

\newcommand{\colim}{\operatorname{colim}}

\newcommand{\End}{\operatorname{End}}
\newcommand{\Ext}{\operatorname{Ext}}

\newcommand{\Hom}{\operatorname{Hom}}
\newcommand{\id}{\operatorname{id}}
\newcommand{\im}{\operatorname{Im}}
\newcommand{\kker}{\operatorname{ker}}
\newcommand{\Ker}{\operatorname{Ker}}
\newcommand{\Perf}{\operatorname{Perf}}

\newcommand{\pr}{\operatorname{pr}}
\newcommand{\qch}{\operatorname{qch}}

\newcommand{\Spec}{\operatorname{Spec}}
\newcommand{\supp}{\operatorname{supp}}

\newcommand{\cC}{\mathcal{C}}

\newcommand{\cH}{\mathcal{H}}

\newcommand{\bN}{\mathbb{N}}

\newcommand{\bZ}{\mathbb{Z}}

\newcommand{\bfk}{\mathbf{k}}

\newcommand{\bfB}{\mathbf{B}}
\newcommand{\bfD}{\mathbf{D}}

\newcommand{\bfK}{\mathbf{K}}

\newcommand{\scrA}{\mathscr{A}}

\newcommand{\scrE}{\mathscr{E}}
\newcommand{\scrF}{\mathscr{F}}
\newcommand{\scrG}{\mathscr{G}}

\newcommand{\scrK}{\mathscr{K}}

\newcommand{\scrO}{\mathscr{O}}

\newcommand{\frakm}{\mathfrak{m}}

\newcommand{\frakU}{\mathfrak{U}}

\title{Twisted categorical generic fiber}
\author[H.~Morimura]{Hayato Morimura}
\address{
SISSA,
via Bonomea 265,
34136
Trieste,
Italy.}
\email{hayato.morimura@ipmu.jp}
\curraddr{
Kavli Institute for the Physics and Mathematics of the Universe (WPI),
University of Tokyo,
5-1-5 Kashiwanoha,
Kashiwa,
Chiba,
277-8583,
Japan.}

\date{}
\begin{document}
\maketitle

\begin{abstract}
After extending Orlov's theorem,
we prove specialization of derived equivalence for flat proper families of Azumaya varieties.
\end{abstract}

\section{Introduction}
When working on a problem in
geometry
and
mathematical physics,
one often enters the realm of noncommutative algebraic geometry.
For instance,
algebraic varieties are tied with in general Azumaya varieties via homological projective duality established by Kuznetsov in
\cite{Kuz07},
which categorifies classical projective duality.
In order to incorporate $B$-fields into homological mirror symmetry,
Kapstin--Orlov suggested to consider the derived category of twisted coherent sheaves
\cite{KP}.

Noncommutativity of an Azumaya variety
$(X, \scrA_X)$
comes from twisting the structure sheaf
$\scrO_X$
by a Brauer class
$[\alpha_X] \in \Br^\prime(X)$,
which corresponds to the equivalence class
$[\scrA_X] \in \Br(X)$
via the canonical inclusion of the Brauer group
$\Br(X)$
to the cohomological Brauer group
$\Br^\prime(X)$.
More explicitly,
there is an equivalence
$\coh(X, \scrA_X) \simeq \coh(X, \alpha_X)$
of categories of
coherent sheaves of right $\scrA_X$-modules
and
coherent $\alpha_X$-twisted sheaves. 
On the other hand,
some $B$-fields define cohomological Brauer classes.

Considering the flexibility of the derived category
which allows us to find relations between nonisomorphic algebraic varieties,
it is natural to ask
whether an Azumaya variety
$(Y, \scrA_Y)$
has its
\emph{geometric realization},
i.e.,
an ordinary algebraic variety
$X$
with
$D^b(X) \simeq D^b(Y, \scrA_Y)$.
Via the derived equivalence
$X$
may be regarded as a resolution of noncommutativity of
$(Y, \scrA_Y)$.
The question makes sense for more general noncommutative spaces.
One of
the earliest
and
the most interesting
examples should be the following.

\begin{theorem}[{\cite[Theorem 5.1]{Cal02}}] \label{thm:5.1}
Let
$f \colon X \to S$
be a generic elliptic Calabi--Yau $3$-fold in the sense of
\cite[Definition 6.1.6]{Cal},
$\varpi \colon J \to S$
its relative Jacobian
and
$\bar{\varpi} \colon \bar{J} \to S$
any analytic small resolution of singularities of
$J$.
For the complement
$U$
of the discriminant locus
$\Delta \subset S$
of
$f$,
let
$\alpha$
be a representative of the element in
$\Br^\prime(J_U)$
corresponding to the base change
$f_U \colon X_U = X \times_S U \to U$
and
$\bar{\alpha}$
its unique extension to
$\bar{J}$.
Then there is an exact $S$-linear equivalence
\begin{align*}
D^b(X) \simeq D^b(\bar{J}, \bar{\alpha}).
\end{align*} 
\end{theorem}

This can be regarded as a special case of our main result,
which gives a sort of sufficient condition for flat proper families of Azumaya varieties to be geometrically realized.

\begin{corollary} \label{cor:GR}
Let
$\tilde{\pi} \colon (X, \scrA_X) \to (S, \scrA_S)$
be a flat proper morphism of Azumaya varieties over
a field
$\bfk$
of characteristic
$0$.
If its generic fiber in the sense of Definition
\pref{dfn:GF}
is geometrically realized by the generic fiber of some flat proper family,
then so is
$\tilde{\pi}$
up to shrinking
$S$.
\end{corollary}

Indeed,
suppose that
$f \colon X \to S$
is any flat projective morphism from a complex $3$-fold
whose generic fiber
$X_\xi$
satisfies
$D^b(X_\xi) \simeq D^b(J_\xi, \alpha_\xi)$
for the image
$\alpha_\xi$
of
$\alpha$
under the canonical map
$\Br^\prime(J_U) \to \Br^\prime(J_\xi)$.
This is the case for the generic elliptic Calabi--Yau $3$-fold from
\pref{thm:5.1}.
By Corollary
\pref{cor:GR}
there is an open subset
$V \subset S$
and
$V$-linear equivalence
$D^b(X_V) \simeq D^b(J_V, \alpha_V)$.
From the argument in
\cite[Section 4]{Morc}
it follows that
$f_V$
is a generic elliptic $3$-fold with
$\omega_{X_V} \cong \scrO_{X_V}$
whose relative Jacobian is
$\varpi_V$.

\begin{remark}
Via Ogg--Shafarevich theory
\cite[Section 4]{Cal},
the fibration
$f_V$
corresponds to a coprime power
$\alpha^d_V$
to the order of
$\alpha_V$
in
$\Br^\prime(J_V)$ 
\cite[Theorem 1.3]{Morb}.
\end{remark}

\begin{remark}
Since
$\bar{J}$
is analytic,
$\bar{\varpi}$
does not have the naive generic fiber. 
However,
one can take the categorical generic fiber in the sense of Definition
\pref{dfn:CGF}.
Extending
\cite{Mora},
it gives a description of the twisted derived category of the generic fiber as a Verdier quotient.
Here in the setting of
\pref{thm:5.1},
by construction one may pass to the generic fiber of
$\varpi_U$.
\end{remark}

\begin{remark}
When
$f$
is originally a generic elliptic Calabi--Yau $3$-fold,
the moduli theoretic argument in
\cite{Cal02}
allows us to obtain the derived equivalence for the whole families.
In other words,
among the lifts to
$D^b(X \times_S \bar{J}, \pr^*_2 \bar{\alpha})$
along the projection to the Verdier quotient of the Fourier--Mukai kernel associated with
$D^b(X_\xi) \simeq D^b(J_\xi, \alpha_\xi)$,
one can find the twisted pseudo-universal sheaf constructed in
\cite[Section 5]{Cal}.
\end{remark}

Corollary
\pref{cor:GR}
is a direct consequence of the following extension of specialization of derived equivalence established in
\cite{Mora}.

\begin{theorem}[Corollary \pref{cor:specialization}] \label{thm:main2}
Let
$\tilde{\pi} \colon (X, \scrA_X) \to (S, \scrA_S), 
\tilde{\pi}^\prime \colon (Y, \scrA_Y) \to (S, \scrA_S)$
be flat proper morphisms of Azumaya $\bfk$-varieties.
Assume that
their generic fibers are derived-equivalent.
Then up to shrinking
$S$
they are $S$-linear derived-equivalent.
In particular,
their closed fibers are derived-equivalent.
\end{theorem}

Our proof is a simple adaptation of the argument in
\cite{Mora}
to twisted case,
for which we need to know
whether any derived equivalence of the generic fibers is of Fourier--Mukai type. 
Extending
\cite{CS07}
to proper case,
we obtain the following stronger result
which should be of independent interest.

\begin{theorem}[cf. {\cite[Theorem 1.1]{CS07}}] \label{thm:main1}
Let
$(X, \scrA_X), (Y, \scrA_Y)$
be smooth proper Azumaya $\bfk$-varieties with
$\scrA_X
\cong
\scrE_X \otimes_{\scrO_X} \scrE^\vee_X,
\scrA_Y
\cong
\scrE_Y \otimes_{\scrO_Y} \scrE^\vee_Y$
for locally free sheaves
$\scrE_X, \scrE_Y$
of finite ranks twisted by
$\alpha_X, \alpha_Y$.
If
$\Phi \colon D^b(X, \alpha_X) \to D^b(Y, \alpha_Y)$
is an exact $\bfk$-linear functor satisfying
\begin{align*}
\Hom_{D^b(Y, \alpha_Y)}(\Phi(\scrF), \Phi(\scrG)[l])
=
0, \
l < 0 
\end{align*} 
for any
$\scrF, \scrG \in \coh(X, \alpha_X)$,
then there exist
$P \in D^b(X \times Y, \alpha^{-1}_X \boxtimes \alpha_Y)$
and
an isomorphism
$\Phi \cong \Phi^{X \to Y}_P$
to the twisted Fourier--Mukai transform with kernel
$P$.
Moreover,
$P$
is uniquely determined up to isomorphism. 
\end{theorem}

This also extends to twisted case Orlov's theorem
\cite{Orl}
for smooth proper varieties proved by
Olander in
\cite{Ola}.
We could have adapted his argument using point-like objects
whose result would require fully faithfullness of
$\Phi$,
which is stronger than the above hypothesis.
Moreover,
we would need a refined version of Gabriel's theorem
\cite[Theorem 3.5]{CG},
instead of
\cite[Theorem 1.1]{Per08},
in adapting the final step of the proof of
\cite[Proposition 2]{Ola}.
However,
the
\emph{pseudo-ample sequence}
tacitly introduced in
\cite{Ola}
is still useful for our purpose.

The significance of the categorical generic fiber might be clearer
when working with
\emph{noncommutative schemes}
over
$S$,
i.e.,
idempotent-complete stable $\infty$-categories equipped with
$\Perf(S)$-actions. 
Over characteristic
$0$,
we may pass to the corresponding pretriangulated dg categories via
\cite[Corollary 5.7]{Coh}.
Also the dg categorical generic fiber
\cite[Definition 1.5]{Mord}
extends to twisted case,
which we expect to be defined for more general noncommutative schemes.
It would be an interesting problem to extend
Corollary
\pref{cor:GR}
to this setting.

\begin{acknowledgements}
The author was supported by SISSA PhD scholarships in Mathematics.
This work was partially supported by
JSPS KAKENHI Grant Number
JP23KJ0341.  
\end{acknowledgements}

\section{Derived categories of smooth proper Azumaya varieties}
In this section,
we review
Azumaya varieties,
twisted sheaves
and
their categorical relation
following
\cite{Kuz06}
and
\cite{Cal}.
Then
we adapt to our setting
pseudo-ample sequences tacitly introduced in
\cite{Ola}
and
some preparatory results from
\cite{CS07}.
Here
and
in the next section,
one may drop the assumption on
$\bfk$
to be of characteristic
$0$.

\subsection{Categories of coherent sheaves on Azumaya varieties}
\begin{definition}[{\cite[Definition D.1]{Kuz06}}]
An
\emph{Azumaya $\bfk$-variety}
is a pair
$(X, \scrA_X)$
of
an algebraic $\bfk$-variety
and
a sheaf
$\scrA_X$
of Azumaya algebras over
$X$.
A
\emph{morphism}
of Azumaya $\bfk$-varieties
$f \colon (X, \scrA_X) \to (Y, \scrA_Y)$
is a pair
$(f_\circ, f_\scrA)$
of
a morphism
$f_\circ \colon X \to Y$
of algebraic $\bfk$-varieties
and
a homomorphism
$f_\scrA \colon f^*_\circ \scrA_Y \to \scrA_X$
of $\scrO_X$-algebras.
\end{definition}

\begin{definition}[{\cite[Definition D.1]{Kuz06}}]
A morphism of Azumaya $\bfk$-varieties
$f \colon (X, \scrA_X) \to (Y, \scrA_Y)$
is
\emph{strict}
if
$\scrA_X \cong f^*_\circ \scrA_Y$
and
$f_\scrA$
is the identity.
A morphism of Azumaya $\bfk$-varieties
$f \colon (X, \scrA_X) \to (Y, \scrA_Y)$
is an
\emph{extension}
if
$X \cong Y$
and
$f_\circ$
is the identity.
\end{definition}

\begin{remark}
Any morphism
$f \colon (X, \scrA_X) \to (Y, \scrA_Y)$
of Azumaya $\bfk$-varieties is canonically decomposed by
an extension
$f^e$
and
a strict morphism
$f^s$
as
\begin{align*}
(X, \scrA_X)
\xrightarrow{f^e}
(X, f^*_\circ \scrA_Y)
\xrightarrow{f^s}
(Y, \scrA_Y).
\end{align*}
\end{remark}

We denote by
$\qch(X, \scrA_X)$
and
$\coh(X, \scrA_X)$
the category of quasicoherent sheaves of right $\scrA_X$-modules
and
its full subcategory of coherent modules.
Up to equivalence,
these are independent of the choice of representatives of the Brauer class
$[\scrA_X] \in \Br(X)$.

\begin{definition}[{\cite[Definition 1.2.1]{Cal}}]
For a \v{C}ech $2$-cocycle
$\alpha_X
=
\{ \alpha_{ijk} \}_{i, j, k \in I}
\in
\check{C}^2_{\text{\'et}}(X, \scrO^*_X)$
with respect to an \'etale cover
$\frakU = \{ U_i \}_{i \in I}$
of an algebraic $\bfk$-variety
$X$,
an
\emph{$\alpha_X$-twisted sheaf}
is a pair
$(\{ \scrF_i \}_{i \in I}, \{ \varphi_{ij} \}_{i, j \in I})$
of collections of
$\scrO_{U_i}$-modules
and
isomorphisms
$\varphi_{ij}
\colon
\scrF_j |_{U_{ij}}
\to
\scrF_i |_{U_{ij}}$
on 
$U_{ij} = U_i \times_X U_j$
satisfying
\begin{align*}
\varphi_{ii}
=
\id, \
\varphi_{ji}
=
\varphi^{-1}_{ij}, \
\varphi_{ij} \circ \varphi_{jk} \circ \varphi_{ki}
=
\alpha_{ijk} \cdot \id 
\end{align*}
for all
$i, j, k \in I$.
It is
\emph{quasicoherent}
if each
$\scrF_i$
is quasicoherent as an $\scrO_{U_i}$-module.
A
\emph{morphism}
of $\alpha_X$-twisted sheaves
$(\{ \scrF_i \}_{i \in I}, \{ \varphi_{ij} \}_{i, j \in I}),
(\{ \scrG_i \}_{i \in I}, \{ \psi_{ij} \}_{i, j \in I})$
is a collection of compatible morphisms
$f_i \colon \scrF_i \to \scrG_i$
with
$\varphi_{ij}, \psi_{ij}$.
\end{definition}

\begin{remark}
In the above definition,
$\scrF_i |_{U_{ij}}$
denote the pullbacks instead of restrictions.
\end{remark}

We denote by
$\qch(X, \alpha_X)$
and
$\coh(X, \alpha_X)$
the category of quasicoherent $\alpha_X$-twisted sheaves on
$X$
and
its full subcategory of coherent sheaves.
Up to equivalence,
these are independent of the choices of
\'etale covers
$\frakU = \{ U_i \}_{i \in I}$
and
representatives of the cohomological Brauer class
$[\alpha_X]
\in
\Br^\prime(X)
=
H^2_{\text{\'et}}(X, \scrO^*_X)$
\cite[Lemma 1.2.3, 1.2.8]{Cal}.

Given an Azumaya $\bfk$-variety
$(X, \scrA_X)$,
the Brauer class
$[\scrA_X]$
is represented by some \v{C}ech $2$-cocycle
$\alpha_X$
via the canonical inclusion
$\Br(X) \subset \Br^\prime(X)$
from
\cite[Theorem 1.1.8]{Cal}.
Then by
\cite[Theorem 1.3.5]{Cal}
there exists a locally free $\alpha_X$-twisted sheaf
$\scrE_X$
of finite rank with
$\scrA_X
\cong
\underline{\End}(\scrE_X)
\cong
\scrE_X \otimes_{\scrO_X} \scrE^\vee_X$. 
For the rest of the paper,
we fix such
$\alpha_X$
and
$\scrE_X$.
The following result allows us to categorically pass from Azumaya varieties to twisted sheaves.

\begin{theorem}[{\cite[Theorem 1.3.7]{Cal}}] \label{thm:1.3.7}
The functors
\begin{align*}
\begin{gathered}
(-) \otimes_{\scrA_X} \scrE_X
\colon
\coh(X, \scrA_X)
\to
\coh(X, \alpha_X), \
(-) \otimes_{\scrO_X} \scrE^\vee_X
\colon
\coh(X, \alpha_X)
\to
\coh(X, \scrA_X)
\end{gathered}
\end{align*}
give mutually inverse $\bfk$-linear equivalences.
\end{theorem}

\begin{remark}
The above result holds without the coherent assumption.
Unless
$X$
is smooth quasiprojective over
$\bfk$,
in general
$\Br(X)$
is strictly smaller than
$\Br^\prime(X)$.
For any Brauer class in the complement
$\Br^\prime(X) \setminus \Br(X)$
there is no locally free twisted sheaf of finite rank representing the class.
Hence categories of coherent twisted sheaves
$\coh(X, \alpha_X)$
cover more variants of
$\coh(X)$
than that of coherent right modules
$\coh(X, \scrA_X)$
over sheaves of Azumaya algebras.
\end{remark}

\subsection{Pseudo-ample sequences}
We denote by
$D^b(X, \scrA_X), D^b(X, \alpha_X)$
the derived categories of
$\coh(X, \scrA_X), \coh(X, \alpha_X)$.
Under the assumptions imposed later,
all usual
derived functors 
and
their relations
will have their twisted counterparts
\cite[Chapter 2]{Cal}.
We use the same symbols to denote derived functors
unless otherwise specified.

In the following three lemmas,
we assume
$\dim X \geq 1$.
Take an affine open cover
$X = \bigcup^M_{i = 1} U_i$
such that
the complements
$D_i = X \setminus U_i$
are proper nontrivial closed subscheme with reduced induced scheme structures for all
$i$.
Consider the system of sheaves
\begin{align} \label{eq:system}
\cdots
\subset
\scrO_X(-n_i D_i)
\subset
\scrO_X(-(n_i - 1) D_i)
\subset
\cdots
\subset
\scrO_X(- D_i).
\end{align}

\begin{lemma}[cf. {\cite[Lemma A]{Ola}}] \label{lem:A}
For any
$\scrF \in \coh(X, \alpha_X)$
there are surjections
\begin{align*}
\bigoplus_i \scrO_X(-n_i D_i)^{\oplus r_i} \otimes_{\scrO_X} \scrE_X
\to
\scrF
\end{align*}
when
$n_i$
are sufficiently large.
Moreover,
these are compatible with the system
\pref{eq:system}.
Namely,
for any
$m_i \geq n_i$
the induced morphism
$\bigoplus_i \scrO_X(-m_i D_i)^{\oplus r_i} \otimes_{\scrO_X} \scrE_X
\to
\scrF$
remains surjective. 
\end{lemma}
\begin{proof}
Choose a surjection
$\varphi_i
\colon
\scrO^{\oplus r_i}_{U_i}
\to
(\scrE^\vee_X \otimes_{\scrO_X} \scrF) |_{U_i}$.
By the formula
\begin{align*}
\Hom_{U_i}(\scrO^{\oplus r_i}_{U_i},
(\scrE^\vee_X \otimes_{\scrO_X} \scrF) |_{U_i})
=
\colim_{n_i}
\Hom_X(\scrO_X(-n_i D_i)^{\oplus r_i},
\scrE^\vee_X \otimes_{\scrO_X} \scrF),
\end{align*}
for sufficiently large
$n_i$
there is a morphism
$\scrO_X(-n_i D_i)^{\oplus r_i}
\to
\scrE^\vee_X \otimes_{\scrO_X} \scrF$
which coincides with
$\varphi_i$
when restricted to
$U_i$. 
\end{proof}

\begin{lemma}[cf. {\cite[Lemma B]{Ola}}] \label{lem:B}
For any
$F \in D^b(X, \alpha_X)$
with
$\cH^0(F) \cong 0$
and
morphism
$\bigoplus_i \scrO_X(-n_i D_i)^{\oplus r_i} \otimes_{\scrO_X} \scrE_X
\to
F$,
there is
$N \in \bN$
such that
when
$m_i \geq N$
the induced morphism
$\bigoplus_i \scrO_X(-m_i D_i)^{\oplus r_i} \otimes_{\scrO_X} \scrE_X
\to
F$
becomes trivial.
\end{lemma}
\begin{proof}
The claim follows from
$\cH^i((\scrE^\vee_X \otimes_{\scrO_X} F) |_{U_i})
=
(\scrE^\vee_X \otimes_{\scrO_X} \cH^i(F)) |_{U_i}$
and
the spectral sequence
\begin{align} \label{eq:spectral}
E^{p,q}_2
=
\Hom_{U_i}(\scrO^{\oplus r_i}_{U_i},
\cH^q((\scrE^\vee_X \otimes_{\scrO_X} F) |_{U_i})[p])
\Rightarrow
\Hom_{U_i}(\scrO^{\oplus r_i}_{U_i},
(\scrE^\vee_X \otimes_{\scrO_X} F) |_{U_i}[p+q]).
\end{align}
\end{proof}

\begin{remark}
In
\cite[Lemma B]{Ola}
the author assumed the vanishing of
$\cH^j(F)$
also for
$j > 0$,
which is irrelevant.
Indeed,
whenever
$q > 0$
the left term of
\pref{eq:spectral}
does not contribute to
$\Hom_{U_i}(\scrO^{\oplus r_i}_{U_i},
(\scrE^\vee_X \otimes_{\scrO_X} F) |_{U_i})$,
as it becomes a negative Ext-group between coherent sheaves.
Moreover,
whenever
$p > 0$
the left term of
\pref{eq:spectral}
does not contribute to
$\Hom_{U_i}(\scrO^{\oplus r_i}_{U_i},
(\scrE^\vee_X \otimes_{\scrO_X} F) |_{U_i})$,
as
$U_i$
are affine.
\end{remark}

Although the last one below does not play any role in the proof of
\pref{thm:main2}
nor  
\pref{thm:main1},
we include it here for completeness.

\begin{lemma}[cf. {\cite[Lemma C]{Ola}}] \label{lem:C}
For any
$\scrF \in \coh(X, \alpha_X)$
there is
$N \in \bN$
such that
when
$n_i \geq N$
we have
\begin{align*}
\Hom_{\coh(X, \alpha_X)}(\scrF, \bigoplus_i \scrO_X(-n_i D_i)^{\oplus r_i} \otimes_{\scrO_X} \scrE_X)
=
0.
\end{align*}
\end{lemma}
\begin{proof}
It suffices to show the vanishing
when the target is
$\scrO_X(-nD) \otimes_{\scrO_X} \scrE_X$
for
a nontrivial effective divisor
$D$
on
$X$
and
sufficiently large
$n$.
By
\pref{lem:A}
we may assume
$\scrF$
ro be
$\scrO_X(-D^\prime) \otimes_{\scrO_X} \scrE_X$
for some effective divisor
$D^\prime$.
Hence it suffices to show the vanishing of 
$H^0(X, \scrO_X(D^\prime - nD) \otimes_{\scrO_X} \scrA_X)$.
Then the argument in
\cite[Lemma C]{Ola}
carries over.
\end{proof}

\subsection{Some preparatory results}
\begin{lemma}[{cf. \cite[Remark 2.1(i)]{CS07}}] \label{lem:homdim}
Let
$(X, \alpha_X)$
be a smooth proper Azumaya $\bfk$-variety.
Then
$\coh(X, \alpha_X)$
has homological dimension
$\dim X$.
\end{lemma}
\begin{proof}
By Serre duality
$\Ext^i_{\coh(X, \alpha_X)}(\scrF, \scrG)$
vanishes for any
$i > \dim X$
and
$\scrF, \scrG \in \coh(X, \alpha_X)$.
Then the claim follows from
\begin{align*}
\Ext^{\dim X}_{\coh(X, \alpha_X)}(\scrE_X, \scrE_X \otimes_{\scrO_X} \omega_X)
\cong
\Ext^0_{\coh(X, \alpha_X)}(\scrE_X, \scrE_X)^\vee
\neq
0.
\end{align*}
\end{proof}

\begin{lemma}[cf. {\cite[Remark 2.1(ii)]{CS07}}] \label{lem:adjoint}
Let
$(X, \scrA_X), (Y, \scrA_Y)$
be smooth proper Azumaya $\bfk$-varieties
and
$\Phi \colon D^b(X, \alpha_X) \to D^b(Y, \alpha_Y)$
an exact $\bfk$-linear functor.
Then
$\Phi$
admits
left
and
right
adjoints
$\Phi^L, \Phi^R$.
\end{lemma}
\begin{proof}
By
\cite[Theorem 1.3]{Per}
any cohomological functor of finite type is representable.
Hence for any
$G \in D^b(Y, \alpha_Y)$
the functor
$\Hom_{D^b(Y, \alpha_Y)}(\Phi(-), G)$
is representable by a unique
$F \in D^b(X, \alpha_X)$.
Setting
$\Psi^R(G) = F$,
via Yoneda lemma we obtain a functor
which is right adjoint to
$\Psi$.
Since by
\pref{lem:homdim} 
all Hom spaces of
$D^b(X, \alpha_X), D^b(Y, \alpha_Y)$
are finite dimensional,
$\Phi^L = S^{-1}_X \circ \Psi^R \circ S_Y$
gives a left adjoint to
$\Psi$.
\end{proof}

\begin{lemma}[cf. {\cite[Proposition 2.4]{CS07}, \cite[Lemma 1]{Ola}}] \label{lem:bdd}
Let
$(X, \scrA_X), (Y, \scrA_Y)$
be smooth proper Azumaya $\bfk$-varieties
and
$\Phi \colon D^b(X, \alpha_X) \to D^b(Y, \alpha_Y)$
an exact $\bfk$-linear functor.
Then
$\Phi$
is bounded,
i.e.,
there exists
$m \in \bN$
satisfying
$\cH^i(\Phi(\scrF)) = 0$
for any
$\scrF \in \coh(X, \alpha_X)$
and
$i \notin [-m, m]$.
\end{lemma}
\begin{proof}
Consider the homological functor
$\Hom_{D^b(Y, \alpha_Y)}(\scrE_Y, \Phi(-))$.
Since by
\cite[Proposition 25]{Per}
the triangulated category
$D^b(X, \alpha)$
has a strong generator,
the argument in
\cite[Tag 0FZ8]{SP}
shows that
$\Hom_{D^b(Y, \alpha_Y)}(\scrE_Y, \Phi(-))$
is bounded.
Then the claim follows from
\begin{align*}
\Hom_{D^b(Y, \alpha_Y)}(\scrE_Y, \Phi(-))
\cong
\Hom_Y(\scrO_Y, \scrE^\vee_Y \otimes_{\scrO_Y} \Phi(-)).
\end{align*}
\end{proof}

\subsection{Point-like objects}
If we tried to extend Orlov's theorem following
\cite{Ola},
then point-like objects of
$D^b(X, \alpha_X)$
would play a central role. 
Although we do not need them to prove
\pref{thm:main2}
nor  
\pref{thm:main1},
we include some explanations for completeness.
Only here,
we assume
$\bfk$
to be algebraically closed,
as otherwise the author does not understand
how the relevant claims in
\cite{Ola}
were justified.

\begin{lemma} \label{lem:canonical}
If
$\bfk$
is algebraically closed,
then the skyscraper sheaf
$k(x)$
for any closed point
$x \in X$
canonically inherits from
$\scrE_X$
the $\alpha_X$-twisted structure.
\end{lemma}
\begin{proof}
The locally free $\alpha_X$-twisted sheaf
$\scrE_X$
carries isomorphisms
$\varphi_{ij}$
on
$X_{ij} = X_i \times_X X_j$
with
\begin{align*}
\varphi_{ij} \circ \varphi_{jk} \circ \varphi_{ki}
=
\alpha_{X, ijk} \cdot \id, \
\alpha_{X, ijk} \in \scrO^*_{X_{ijk}}(X_{ijk}).
\end{align*}
On the stalk at any closed point 
$x \in X$,
we have the induced isomorphisms
$\varphi_{ij, x}$,
which are given by invertible
$r \times r$
matrices with coefficients in
$\scrO_{X_{ijk}, x}$,
satisfying
\begin{align*}
\varphi_{ij, x} \circ \varphi_{jk, x} \circ \varphi_{ki, x}
=
\alpha_{X, ijk, x} \cdot \id_{\scrE_{X, x}}, \
\alpha_{X, ijk, x} \in \scrO^*_{X_{ijk}, x} = k(x) \cong \bfk.
\end{align*}
When
$\bfk$
is algebraically closed,
the collection
$(\{ k(x) \}, \{ (\det \varphi_{ij, x})^{1/r} \})$
is well defined
and
gives an $\alpha_X$-twisted sheaf on
$X$.
\end{proof}

\begin{remark}
When
$\bfk$
is algebraically closed,
$\Br^\prime(\bfk)$
is trivial
and
by
\cite[Lemma 1.2.8]{Cal}
there is an equivalence
\begin{align*}
\coh(\Spec k(x), \alpha_{X, x})
\simeq
\coh(\Spec k(x)),
\end{align*}
via
which
$(\{ k(x) \}, \{ (\det \varphi_{ij, x})^{1/r} \})$
corresponds to the ordinary skyscraper sheaf
$k(x)$. 
\end{remark}

\begin{lemma} \label{lem:span}
Let
$(X, \scrA_X)$
be a smooth Azumaya variety over an algebraically closed
$\bfk$.
Then the objects
$k(x)$
for all closed points
$x \in X$
form a spanning class of
$D^b(X, \alpha_X)$.
\end{lemma}
\begin{proof}
The proof goes parallel to
\cite[Proposition 3.17]{Huy}.
It suffices to show that
for any nontrivial
$F \in D^b(X, \alpha_X)$
there exist closed points
$x_1, x_2$
and
$i_1, i_2 \in \bZ$
such that
\begin{align*}
\Hom_{D^b(X, \alpha_X)}(F, k(x_1)[i_1]) \neq 0, \
\Hom_{D^b(X, \alpha_X)}(k(x_2)[i_2], F) \neq 0.
\end{align*}
Consider the spectral sequence
\begin{align*}
E^{p, q}_2
=
\Ext^p_{D^b(X, \alpha_X)}(\cH^{-q}(F), k(x))
\Rightarrow
\Ext^{p+q}_{D^b(X, \alpha_X)}(F, k(x)).
\end{align*}
Take a maximal
$m \in \bZ$
with
$\cH^m(F) \neq 0$.
Then
$E^{0, -m}_\infty = E^{0, -m}_2$,
since by
\cite[Proposition 2.1.8]{Cal}
negative Ext-groups between coherent $\alpha_X$-twisted sheaves vanish.
Hence we have
\begin{align*}
\Hom_{D^b(X, \alpha_X)}(F, k(x)[-m])
=
\Hom_{D^b(X, \alpha_X)}(\cH^m(F), k(x))
\cong
\Hom_{D^b(\Spec k(x))}(\cH^m(F)_x / \frakm_x, k(x))
\neq
0
\end{align*}
for some closed point
$x \in \supp \cH^m(F)$.
The rest follows from Serre duality.
\end{proof}

\section{Orlov's theorem for smooth proper Azumaya varieties}
In this section,
we extend
Orlov's theorem for smooth proper varieties
\cite[Theorem]{Ola}
to twisted case
and
at the same time its refinement for smooth projective Azumaya varieties
\cite[Theorem 1.1]{CS07}
to proper case.
Here,
we adapt the proof of
\cite[Theorem 1.1]{CS07}
to our setting,
as it requires a weaker assumption.

\subsection{Boundedness and pseudo-ample sequences}
\begin{lemma}[cf. {\cite[Lemma 2.5]{CS07}}] \label{lem:2.5}
Let
$(X, \scrA_X), (Y, \scrA_Y)$
be smooth proper Azumaya $\bfk$-varieties
and
$P \in D^b(X \times Y, \alpha^{-1}_X \boxtimes \alpha_Y)$.
Then we have
$\cH^l(P) = 0$
if and only if
\begin{align*}
\Phi^{X \to Y}_P(\bigoplus_i \scrO_X(n_i D_i)^{\oplus r_i} \otimes_{\scrO_X} \scrE_X) = 0 \
\text{for any} \
n_i \gg 0.
\end{align*}
\end{lemma}
\begin{proof}
Let
$A
=
P \otimes \pr^*_Y \scrE^\vee_Y \otimes \pr^*_X \scrE_X$.
When
$\cH^p(A) \neq 0$,
we have
\begin{align*}
R^q \pr_{Y *}(\cH^p(A) \otimes \bigoplus_i \pr^*_X (\scrO_X(n_i D_i)^{\oplus r_i}))
=
0 \
\text{for any} \
n_i \gg 0
\end{align*}
if and only if
$q > 0$.
To see this,
we may assume
$Y$
to be affine.
Then we have
\begin{align*}
R^q \pr_{Y *}(\cH^p(A) \otimes \bigoplus_i \pr^*_X (\scrO_X(n_i D_i)^{\oplus r_i})
\cong
\Hom_X(\bigoplus_i (\scrO_X(-n_i D_i)^{\oplus r_i}, \pr_{X *} \cH^p(A)[q]),
\end{align*}
whose right term vanishes by
\pref{lem:B}
when
$n_i \gg 0$. 
For the rest,
the argument in
\cite[Lemma 2.5]{CS07}
works without modification.
\end{proof}

\subsection{Convolutions and isomorphisms of functors}
\begin{lemma}[cf. {\cite[Proposition 3.7]{CS07}}] \label{lem:3.7}
Let
$(X, \scrA_X)$
be a smooth proper Azumaya $\bfk$-variety.
We denote by
$\bfB$
the full subcategory of
$D^b(X, \alpha_X)$
with the objects of the form
$\bigoplus_i \scrO_X(n_i D_i)^{\oplus r_i} \otimes_{\scrO_X} \scrE_X$.
Let
$\Phi_1, \Phi_2 \colon D^b(X, \alpha_X) \to \bfD$
be exact $\bfk$-linear functors to a triangulated category
$\bfD$
such that
\begin{itemize}
\item[(i)]
there exists an isomorphism
$h \colon \Phi_2 |_\bfB \to \Phi_1 |_\bfB$
of functors ;
\item[(ii)]
$\Hom_\bfD(\Phi_1(\scrF), \Phi_1(\scrG)[l]) = 0$
for any
$\scrF, \scrG \in \coh(X, \alpha_X)$
and
$l <0$ ;
\item[(iii)]
$\Phi_1$
admits a left adjoint
$\Phi^L_1$.
\end{itemize}
Then
$h$
extends to an isomorphism
$\tilde{h} \colon \Phi_2 \to \Phi_1$
of functors.
\end{lemma}
\begin{proof}
Given
$\scrF \in \coh(X, \alpha_X)$,
we will construct an isomorphism
$h(\scrF) \colon \Phi_2(\scrF) \to \Phi_1(\scrF)$
and
then extend it to
$\tilde{h}$.
By
\pref{lem:A}
one can take a locally free resolution
\begin{align} \label{eq:resolution1}
\cdots
\xrightarrow{d_{j+1}}
Q_j
\xrightarrow{d_j}
Q_{j-1}
\xrightarrow{d_{j-1}}
\cdots
\xrightarrow{d_1}
Q_0
\xrightarrow{d_0}
\scrF
\to
0
\end{align}
of
$\scrF$
with
$Q_j
=
\bigoplus_{i} \scrO_X(-n_{i, j} D_i)^{\oplus r_{i, j}} \otimes_{\scrO_X} \scrE_X$.
Consider the bounded complex
\begin{align} \label{eq:R_m}
Q_m
\xrightarrow{d_m}
Q_{m-1}
\xrightarrow{d_{m-1}}
\cdots
\xrightarrow{d_1}
Q_0
\end{align}
for a fixed integer
$m > \dim X$.
According to
\cite[Remark 3.6]{CS07},
there is a unique up to isomorphism right convolution
$(d_0, 0)
\colon
Q_0
\to
\scrF \oplus \scrK_m[m]$
of
\pref{eq:R_m}
with
$\scrK_m = \Ker (d_m)$.
Then by
\cite[Remark 3.1]{CS07}
the bounded complexes
\begin{align} \label{eq:Phi(R_m)}
\Phi_k(Q_m)
\xrightarrow{\Phi_k(d_m)}
\Phi_k(Q_{m-1})
\xrightarrow{\Phi_k(d_{m-1})}
\cdots
\xrightarrow{\Phi_k(d_1)}
\Phi_k(Q_0)
\end{align}
for
$k = 1, 2$
admit right convolutions
$(\Phi_k(d_0), 0)
\colon
\Phi_k(Q_0)
\to
\Phi_k(\scrF \oplus \scrK_m[m])$.
By
the condition
(i),
(ii)
and
\cite[Lemma 3.2]{CS07}
such convolutions are unique up to isomorphism.
Moreover,
we have
\begin{align*}
\Hom_\bfD(\Phi_2(Q_b), \Phi_1(\scrF)[l])
=
0
=
\Hom_\bfD(\Phi_2(Q_c), \Phi_1(Q_a)[l])
\end{align*}
for
any
$a, b, c \in [0, m]$
and
$l <0$.
Hence one can apply
\cite[Lemma 3.3]{CS07}
to obtain a unique isomorphism
$h(\scrF) \colon \Phi_2(\scrF) \to \Phi_1(\scrF)$,
which makes the diagram
\begin{align*}
\begin{gathered}
\xymatrix{
\Phi_2(Q_m) \ar[r]^-{\Phi_2(d_m)} \ar_-{h_m}[d] & \Phi_2(Q_{m-1}) \ar[r]^-{\Phi_2(d_{m-1})} \ar_-{h_{m-1}}[d] & \cdots \ar[r]^-{\Phi_2(d_1)} & \Phi_2(Q_0) \ar[r]^-{\Phi_2(d_0)} \ar_-{h_0}[d] & \Phi_2(\scrF) \ar_-{h(\scrF)}[d] \\
\Phi_1(Q_m) \ar[r]^-{\Phi_1(d_m)} & \Phi_1(Q_{m-1}) \ar[r]^-{\Phi_1(d_{m-1})} & \cdots \ar[r]^-{\Phi_1(d_1)} & \Phi_1(Q_0) \ar[r]^-{\Phi_2(d_0)} & \Phi_1(\scrF)
}
\end{gathered}
\end{align*}
commute
and
does not depend on the choice of
$m$.

The argument in
\cite[Proposition 3.7]{CS07}
shows that
the definition of
$h(\scrF)$
does not depend on the choice of the resolution
\pref{eq:resolution1},
if for any other resolution
\begin{align} \label{eq:resolution2}
\cdots
\xrightarrow{d^\prime_{j+1}}
Q^\prime_j
\xrightarrow{d^\prime_j}
Q^\prime_{j-1}
\xrightarrow{d^\prime_{j-1}}
\cdots
\xrightarrow{d^\prime_1}
Q^\prime_0
\xrightarrow{d^\prime_0}
\scrF
\to
0
\end{align}
and
$a \geq 0$
there exist a third resolution
\begin{align} \label{eq:resolution3}
\cdots
\xrightarrow{d^{\prime \prime}_{j+1}}
Q^{\prime \prime}_j
\xrightarrow{d^{\prime \prime}_j}
Q^{\prime \prime}_{j-1}
\xrightarrow{d^{\prime \prime}_{j-1}}
\cdots
\xrightarrow{d^{\prime \prime}_1}
Q^{\prime \prime}_0
\xrightarrow{d^{\prime \prime}_0}
\scrF
\to
0
\end{align}
and
compatible morphisms
$s_a \colon Q^{\prime \prime}_a \to Q_a,
t_a \colon Q^{\prime \prime}_a \to Q^\prime_a$.
One can construct
\pref{eq:resolution3}
inductively as follows.
First,
by
\pref{lem:A}
there exists a surjection
$d^{\prime \prime}_0 \colon Q^{\prime \prime}_0 \to \scrF$
factorizing thorough
$d_0, d^\prime_0$
when
$n^{\prime \prime}_{i, 0} \gg 0$.
Suppose that
$Q^{\prime \prime}_j, d^{\prime \prime}_j, s_j$
and
$t_j$
are defined.
Applying
\pref{lem:A}
to
$\Ker (d^{\prime \prime}_j)$,
one obtains a morphism
$d^{\prime \prime}_{j+1}
\colon
Q^{\prime \prime}_{j+1}
\twoheadrightarrow
\Ker (d^{\prime \prime}_j)
\hookrightarrow
Q^{\prime \prime}_j$.
Since by
\pref{lem:B}
the induced maps
\begin{align*}
\begin{gathered}
\Hom_{\coh(X, \alpha)}(Q^{\prime \prime}_{j+1}, Q_{j+1})\to
\Hom_{\coh(X, \alpha)}(Q^{\prime \prime}_{j+1}, \im (d_{j+1})), \\
\Hom_{\coh(X, \alpha)}(Q^{\prime \prime}_{j+1}, Q^\prime_{j+1})\to
\Hom_{\coh(X, \alpha)}(Q^{\prime \prime}_{j+1}, \im (d^\prime_{j+1}))
\end{gathered}
\end{align*}
are surjective,
the morphisms
$s_j \circ d^{\prime \prime}_{j+1}, t_j \circ d^{\prime \prime}_{j+1}$
respectively factorize through
$d_{j+1}, d^\prime_{j+1}$.
As for 
the functoriality of
$h(\scrF)$
and
its extension to
$\tilde{h}$,
the argument in
\cite[Proposition 3.7]{CS07}
works without modification. 
\end{proof}


\subsection{Proof of \pref{thm:main1}}
Regard the pushforward 
$\scrO_\Delta$
of
$\scrO_X$
along the diagonal embedding as an $\alpha^{-1}_X \boxtimes \alpha_X$-twisted sheaf on
$X \times X$
in a natural way.
By the same argument as in
\pref{lem:A},
one sees that
there is a surjection
\begin{align*}
\bigoplus_i
(\scrO_X(-n_i D_i)^{\oplus r_i}
\boxtimes
\scrO_X(-n_i D_i)^{\oplus r_i})
\otimes_{\scrO_{X \times X}}
(\scrE^\vee_X \boxtimes \scrE_X)
\to
\scrO_\Delta
\end{align*}
for
$n_i \gg 0$
and
some
$r_i$.
Hence
$\scrO_\Delta$
admits a locally free resolution
\begin{align*}
\cdots
\to
A_j \boxtimes B_j
\xrightarrow{\delta_j}
A_{j-1} \boxtimes B_{j-1}
\xrightarrow{\delta_{j-1}}
\cdots
\xrightarrow{\delta_1}
A_0 \boxtimes B_0
\xrightarrow{\delta_0}
\scrO_\Delta
\to
0
\end{align*}
with
$A_j \in \coh(X, \alpha^{-1}_X), B_j \in \coh(X, \alpha_X)$
as we have
\begin{align*}
(\scrO_X(-n_i D_i)
\boxtimes
\scrO_X(-n_i D_i))
\otimes_{\scrO_{X \times X}}
(\scrE^\vee_X \boxtimes \scrE_X)
\cong
(\scrO_X(-n_i D_i) \otimes_{\scrO_X} \scrE^\vee_X)
\boxtimes
(\scrO_X(-n_i D_i) \otimes_{\scrO_X} \scrE_X).
\end{align*}
We fix such a resolution
once
and
for all.

Since
$\Phi$
is bounded by
\pref{lem:bdd},
we may assume that
$\cH^j(\Phi(\scrF)) = 0$
for any
$\scrF \in \coh(X, \alpha_X)$
and
$j \notin [-N, 0]$
for some
$N \in \bN$.
When
$m > \dim X + \dim Y + N$,
we define the bounded complexes
$C_m \in D^b(X \times X, \alpha^{-1}_X \boxtimes \alpha_X)$
and
$\tilde{C}_m \in D^b(X \times Y, \alpha^{-1}_X \boxtimes \alpha_Y)$
respectively as
\begin{align*}
A_m \boxtimes B_m
\xrightarrow{\delta_m}
\cdots
\xrightarrow{\delta_1}
A_0 \boxtimes B_0, \
A_m \boxtimes \Phi(B_m)
\xrightarrow{\tilde{\delta}_m}
\cdots
\xrightarrow{\tilde{\delta}_1}
A_0 \boxtimes \Phi(B_0).
\end{align*}
Here,
$\tilde{\delta}_j$
denote the images of
$\delta_j$
under the maps induced by
$1_{D^b(X, \alpha_X)} \otimes \Phi$.
The argument in
\cite[Remark 3.6]{CS07}
shows that
$C_m$
has a right convolution
$(\delta_0, 0)
\colon
A_0 \boxtimes B_0
\to
\scrO_\Delta \oplus \scrK_m[m]$
for
$\scrK_m = \Ker(\delta_m)$
when
$m > 2 \dim X$.
By
\cite[Lemma 3.2]{CS07}
the convolution is unique up to isomorphism.
Also
$\tilde{C}_m$
has a unique up to isomorphism right convolution
$\tilde{\delta}^\prime_{0, m}
\colon
A_0 \boxtimes \Phi(B_0)
\to
\scrG_m$,
as the assumption on
$\Phi$
implies that
$\tilde{C}_m$
satisfies the hypothesis of
\cite[Lemma 3.2]{CS07}.

We denote by
$\bfK_m$
the full subcategory of
$\coh(X, \alpha_X)$
with the objects locally free $\alpha_X$-twisted sheaves
$\scrE$
such that
$H^q(X, A_j \otimes_{\scrO_X} \scrE) = 0$
for
$q > 0$
and
$0 \leq j \leq \dim X + m$.
From
\pref{lem:B}
it follows
$\bigoplus_i \scrO_X(n_i D_i)^{\oplus r_i} \otimes_{\scrO_X} \scrE
\in
\bfK_m$
for
any locally free
$\scrE \in \coh(X, \alpha_X)$
and
$r_i$
when
$n_i \gg 0$.
If
$\scrE \in \bfK_m$
then the functors
\begin{align*}
R \pr_{2 *} (- \otimes_{\scrO_{X \times X}} \pr^*_1 \scrE), \
R \pr_{Y *} (- \otimes_{\scrO_{X \times Y}} \pr^*_X \scrE)
\end{align*}
send
$C_m, \tilde{C}_m$
respectively to
\begin{align*}
\begin{gathered}
C_{m, \scrE}
=
H^0(X, A_m \otimes_{\scrO_X} \scrE) \otimes_\bfk B_m
\xrightarrow{\delta_{m, \scrE}}
\cdots
\xrightarrow{\delta_{1, \scrE}}
H^0(X, A_0 \otimes_{\scrO_X} \scrE) \otimes_\bfk B_0, \\
\tilde{C}_{m, \scrE}
=
H^0(X, A_m \otimes_{\scrO_X} \scrE) \otimes_\bfk \Phi(B_m)
\xrightarrow{\tilde{\delta}_{m, \scrE}}
\cdots
\xrightarrow{\tilde{\delta}_{1, \scrE}}
H^0(X, A_0 \otimes_{\scrO_X} \scrE) \otimes_\bfk \Phi(B_0)
\end{gathered}
\end{align*}
which have unique up to isomorphism right convolutions
\begin{align*}
\begin{gathered}
(\delta_{0, \scrE}, 0)
\colon
H^0(X, A_0 \otimes_{\scrO_X} \scrE) \otimes_\bfk B_0
\to
\scrE \oplus \scrK_{m, \scrE}[m], \
\scrK_{m, \scrE} = \Ker(\delta_{m, \scrE}), \\
\tilde{\delta}^\prime_{0, m, \scrE}
\colon
H^0(X, A_0 \otimes_{\scrO_X} \scrE) \otimes_\bfk \Phi(B_0)
\to
R \pr_{Y *}(\scrG_m \otimes \pr^*_X \scrE)
=
\Phi^{X \to Y}_{\scrG_m}(\scrE).
\end{gathered}
\end{align*}
Since
$\tilde{C}_{m, \scrE}$
can be identified with
$\Phi(C_{m, \scrE})$,
we obtain another right convolution
\begin{align*}
(\Phi(\delta_{0, \scrE}), 0)
\colon
H^0(X, A_0 \otimes_{\scrO_X} \scrE) \otimes_\bfk \Phi(B_0)
\to
\Phi(\scrE) \oplus \Phi(\scrK_{m, \scrE})[m].
\end{align*}
Hence uniqueness of convolutions yields
$\Phi^{X \to Y}_{\scrG_m}(\scrE)
\cong
\Phi(\scrE) \oplus \Phi(\scrK_{m, \scrE})[m]$.
Taking
$m, N$
sufficiently large,
by
\pref{lem:bdd}
one finds
$P_m, P^\prime_m
\in
D^b(X \times Y, \alpha^{-1}_X \boxtimes \alpha_Y)$
such that
$\scrG_m \cong P_m \oplus P^\prime_m$
with
$\cH^j(P_m) = 0$
if
$j \notin [-N, 0]$
and
$\cH^j(P^\prime_m) = 0$
if
$j \notin [-m-N, -m]$.
As for the fact that
the kernel of
$\Phi$,
if exists,
is unique up to isomorphism
and
the restrictions
$\Phi |_{\bfK_m}, \Phi^{X \to Y}_{P_m} |_{\bfK_m}$
are isomorphic
when
$m > \dim X + \dim Y + N$,
the arguments in
\cite[Section 4.3, 4.4]{CS07}
work without modification.

Finally,
in order to apply
\pref{lem:3.7},
we extend the isomorphism
$\Phi |_{\bfK_m} \cong \Phi^{X \to Y}_{P_m} |_{\bfK_m}$
to the full subcategory
$\bfB \subset \coh(X, \alpha_X)$.
Introduce a partial order in
$\bZ^M$
defined as
\begin{align*}
(n_1, \ldots, n_M)
\leq
(m_1, \ldots, m_M)
\Leftrightarrow
n_j \leq m_j \ \text{for all} \ 1 \leq j \leq M.
\end{align*}
In the sequel,
we denote
$\bigoplus_i \scrO_X(n_i D_i)^{\oplus r_i} \otimes_{\scrO_X} \scrE_X$
by
$Q^+_{n_i, r_i}$.
We will construct isomorphisms
\begin{align*}
\varphi_{(n_i, r_i)}
\colon
\Phi(Q^+_{n_i, r_i})
\to
\Phi^{X \to Y}_{P_m}(Q^+_{n_i, r_i})
\end{align*}
for all
$(n_1, \ldots, n_M, r_1, \ldots, r_M)
\in
\bZ^M \times \bZ^M$
satisfying
\begin{align} \label{eq:compatibility}
\Phi^{X \to Y}_{P_m}(\gamma) \circ \varphi_{(m_i, s_i)}
=
\varphi_{(m^\prime_i, s^\prime_i)} \circ \Phi(\gamma)
\end{align}
for any
$(m_1, \ldots, m_M, s_1, \ldots, s_M),
(m^\prime_1, \ldots, m^\prime_M, s^\prime_1, \ldots, s^\prime_M)
\in
\bZ^M \times \bZ^M$
and
morphism
$\gamma
\colon
Q^+_{m_i, s_i}
\to
Q^+_{m^\prime_i, s^\prime_i}$
in
$\bfB$.
Due to
\pref{lem:B},
one finds
$(n_1, \ldots, n_M) \in \bZ^M$
such that
$Q^+_{m_i, s_i} \in \bfK_m$
for any
$(m_1, \ldots, m_M) > (n_1, \ldots, n_M)$
and
$(s_1, \ldots, s_M)$.
Set
$\varphi_{(m_i, s_i)}
=
\varphi(Q^+_{m_i, s_i})^{-1}$
where
\begin{align*}
\varphi(Q^+_{m_i, s_i})
\colon
\Phi^{X \to Y}_{P_m}(Q^+_{m_i, s_i})
\to
\Phi(Q^+_{m_i, s_i})
\end{align*}
denotes the functorial isomorphism from
\cite[Section 4.4]{CS07},
which
satisfies
\begin{align*}
(\Phi(\delta_{0, Q^+_{m_i, s_i}}), 0)
=
\varphi(Q^+_{m_i, s_i})
\circ
\tilde{\delta}^\prime_{0, m, Q^+_{m_i, s_i}}
\end{align*}
and
is induced by the above isomorphisms
\begin{align*}
\Phi^{X \to Y}_{\scrG_m}(Q^+_{m_i, s_i})
\cong
\Phi(Q^+_{m_i, s_i}) \oplus \Phi(\scrK_{m, Q^+_{m_i, s_i}})[m], \
\scrG_m \cong P_m \oplus P^\prime_m.
\end{align*}
Then
\pref{eq:compatibility}
holds
when
$(m_1, \ldots, m_M),
(m^\prime_1, \ldots, m^\prime_M)
>
(n_1, \ldots, n_M)$.

Now,
we proceed by descending induction on the partial order of
$\bZ^M$.
By
\cite[Lemma A]{Ola}
there is a locally free resolution
\begin{align*}
\bigoplus_i \scrO_X(-n_{i, \dim X} D_i)^{\oplus r_{i, \dim X}}
\to
\cdots
\to
\bigoplus_i \scrO_X(-n_{i, 1} D_i)^{\oplus r_{i, 1}}
\to
\bigoplus_i \scrO_X(-n_{i, 0} D_i)^{\oplus r_{i, 0}}
\to
\scrO_X
\to
0
\end{align*}
with
$(n_{1, k}, \ldots, n_{M, k})
\leq
(n_{1, l}, \ldots, n_{M, l})$
for
$k \leq l$.
Applying
\cite[Lemma 3.2]{CS07}
to the induced complex
\begin{align*}
Q^+_{n_{i, 0} + n_i, r_{i, 0} + r_i}
\xrightarrow{\rho^{(n_i, r_i)}_{\dim X}}
Q^+_{n_{i, 1} + n_i, r_{i, 1} + r_i}
\xrightarrow{\rho^{(n_i, r_i)}_{\dim X - 1}}
\cdots
\xrightarrow{\rho^{(n_i, r_i)}_1}
Q^+_{n_{i, \dim X} + n_i, r_{i, \dim X} + r_i},
\end{align*}
we obtain a unique up to isomorphism left convolution
$\rho^{(n_i, r_i)}_{\dim X + 1}
\colon
Q^+_{n_i, r_i}
\to
Q^+_{n_{i, 0} + n_i, r_{i, 0} + r_i}$.
The inductive hypothesis implies that
\begin{align*}
\begin{gathered}
\xymatrix@C=42pt@R=25pt{
\Phi(Q^+_{n_{i, 0} + n_i, r_{i, 0} + r_i}) \ar[r]^-{\Phi(\rho^{(n_i, r_i)}_{\dim X})} \ar_-{\varphi_{(n_{i, 0} + n_i, r_{i, 0} + r_i)}}[d]
&
\Phi(Q^+_{n_{i, 1} + n_i, r_{i, 1} + r_i}) \ar[r]^-{\Phi(\rho^{(n_i, r_i)}_{\dim X - 1})} \ar_-{\varphi_{(n_{i, 1} + n_i, r_{i, 1} + r_i)}}[d]
&
\cdots \ar[r]^-{\Phi(\rho^{(n_i, r_i)}_1)}
&
\Phi(Q^+_{n_{i, \dim X} + n_i, r_{i, \dim X} + r_i}) \ar_-{\varphi_{(n_{i, \dim X} + n_i, r_{i, \dim X} + r_i)}}[d] \\
\Phi^{X \to Y}_{P_m}(Q^+_{n_{i, 0} + n_i, r_{i, 0} + r_i}) \ar[r]^-{\Phi^{X \to Y}_{P_m}(\rho^{(n_i, r_i)}_{\dim X})}
&
\Phi^{X \to Y}_{P_m}(Q^+_{n_{i, 1} + n_i, r_{i, 1} + r_i}) \ar[r]^-{\Phi^{X \to Y}_{P_m}(\rho^{(n_i, r_i)}_{\dim X - 1})}
&
\cdots \ar[r]^-{\Phi^{X \to Y}_{P_m}(\rho^{(n_i, r_i)}_1)}
&
\Phi^{X \to Y}_{P_m}(Q^+_{n_{i, \dim X} + n_i, r_{i, \dim X} + r_i})
}
\end{gathered}
\end{align*}
is an isomorphism of complexes satisfying the hypothesis of
\cite[Lemma 3.3]{CS07}.
Hence there is a unique isomorphism
$\varphi_{(n_i, r_i)}$
which makes the diagram
\begin{align*}
\begin{gathered}
\xymatrix@C=40pt{
\Phi(Q^+_{n_i, r_i}) \ar[r]^-{\Phi(\rho^{(n_i, r_i)}_{\dim X + 1})} \ar_-{\varphi_{(n_i, r_i)}}[d]
&
\Phi(Q^+_{n_{i, 0} + n_i, r_{i, 0} + r_i}) \ar_-{\varphi_{(n_{i, 0} + n_i, r_{i, 0} + r_i)}}[d] \\
\Phi^{X \to Y}_{P_m}(Q^+_{n_i, r_i}) \ar[r]^-{\Phi^{X \to Y}_{P_m}(\rho^{(n_i, r_i)}_{\dim X + 1})}
&
\Phi^{X \to Y}_{P_m}(Q^+_{n_{i, 0} + n_i, r_{i, 0} + r_i})
}
\end{gathered}
\end{align*}
commute.
Moreover,
setting
\begin{align*}
\tilde{\gamma}_j
=
\varphi_{(n_{i, \dim X - j} + m^\prime_i, r_{i, \dim X - j} + s^\prime_i)}
\circ
\Phi(\gamma_j)
=
\Phi^{X \to Y}_{P_m}(\gamma_j)
\circ
\varphi_{(n_{i, \dim X - j} + m_i, r_{i, \dim X - j} + s_i)}
\end{align*}
for any
$(m_1, \ldots, m_M),
(m^\prime_1, \ldots, m^\prime_M)
>
(n_1, \ldots, n_M)$
and
morphism
\begin{align*}
\gamma_j
\colon
Q^+_{n_{i, \dim X - j} + m_i, r_{i, \dim X - j} + s_i}
\to
Q^+_{n_{i, \dim X - j} + m^\prime_i, r_{i, \dim X - j} + s^\prime_i}
\end{align*}
in
$\bfB$,
we obtain a morphism
\begin{align*}
\begin{gathered}
\xymatrix@C=40pt{
\Phi(Q^+_{n_{i, 0} + m_i, r_{i, 0} + s_i}) \ar[r]^-{\Phi(\rho^{(m_i, s_i)}_{\dim X})} \ar_-{\tilde{\gamma}_{\dim X}}[d]
&
\Phi(Q^+_{n_{i, 1} + m_i, r_{i, 1} + s_i}) \ar[r]^-{\Phi(\rho^{(m_i, s_i)}_{\dim X - 1})} \ar_-{\tilde{\gamma}_{\dim X - 1}}[d]
&
\cdots \ar[r]^-{\Phi(\rho^{(m_i, s_i)}_1)}
&
\Phi(Q^+_{n_{i, \dim X} + m_i, r_{i, \dim X} + s_i}) \ar_-{\tilde{\gamma}_0}[d] \\
\Phi^{X \to Y}_{P_m}(Q^+_{n_{i, 0} + m^\prime_i, r_{i, 0} + s^\prime_i}) \ar[r]^-{\Phi^{X \to Y}_{P_m}(\rho^{(m^\prime_i, s^\prime_i)}_{\dim X})}
&
\Phi^{X \to Y}_{P_m}(Q^+_{n_{i, 1} + m^\prime_i, r_{i, 1} + s^\prime_i}) \ar[r]^-{\Phi^{X \to Y}_{P_m}(\rho^{(m^\prime_i, s^\prime_i)}_{\dim X - 1})}
&
\cdots \ar[r]^-{\Phi^{X \to Y}_{P_m}(\rho^{(m^\prime_i, s^\prime_i)}_1)}
&
\Phi^{X \to Y}_{P_m}(Q^+_{n_{i, \dim X} + m^\prime_i, r_{i, \dim X} + s^\prime_i})
}
\end{gathered}
\end{align*}
of complexes satisfying the hypothesis of
\cite[Lemma 3.3]{CS07}.
Hence there is a unique morphism
$\Phi(Q^+_{m_i, s_i}) \to \Phi^{X \to Y}_{P_m}(Q^+_{m^\prime, s^\prime_i})$
which makes the diagram
\begin{align*}
\begin{gathered}
\xymatrix@C=40pt{
\Phi(Q^+_{m_i, s_i}) \ar[r]^-{\Phi(\rho^{(m_i, s_i)}_{\dim X + 1})} \ar_{}[d]
&
\Phi(Q^+_{n_{i, 0} + m_i, r_{i, 0} + s_i}) \ar_-{\tilde{\gamma}_{\dim X}}[d] \\
\Phi^{X \to Y}_{P_m}(Q^+_{n_i, r_i}) \ar[r]^-{\Phi^{X \to Y}_{P_m}(\rho^{(m^\prime_i, s^\prime_i)}_{\dim X + 1})}
&
\Phi^{X \to Y}_{P_m}(Q^+_{n_{i, 0} + m^\prime_i, r_{i, 0} + s^\prime_i})
}
\end{gathered}
\end{align*}
commute.
Since both
$\Phi^{X \to Y}_{P_m}(\gamma) \circ \varphi_{(m_i, s_i)}$
and
$\varphi_{(m^\prime_i, s^\prime_i)} \circ \Phi(\gamma)$
satisfy this property,
\pref{eq:compatibility}
holds.

\section{Specialzation of twisted Fourier--Mukai functors}
In this section,
we extend
\cite[Corollary 6.9]{Mora}
to twisted case
after introducing the generic fiber of a flat proper families of Azumaya varieties in both
geometric
and
categorical
ways.

\subsection{The generic fiber}
Given a morphism
$\tilde{\pi}
\colon
(X, \scrA_X)
\to
(S, \scrA_S)$
of Azumaya $\bfk$-varieties,
consider the composition
$\pi
\colon
(X, \scrA_X)
\to
S$
with the structure morphism
$(S, \scrA_S) \to S$.
Let
$\iota_\xi \colon \Spec K \to S$
be the canonical morphism
where
$K$
denotes the function filed
$k(S)$,
which coincides with the stalk
$\scrO_{S, \xi}$
at the generic point
and
the field of fractions
$Q(R)$
for any affine open subset
$U = \Spec R \subset S$.
Since
$\iota_\xi$
is strict,
by
\cite[Lemma D.37]{Kuz06}
the pair
$(X \times_S \Spec K, \pr^*_{1, \circ} \scrA_X)$
is a fiber product of
$(X, \scrA_X)$
and
$\Spec K$
over
$S$.
In summary,
we have the following pullback diagram
\begin{align*}
\begin{gathered}
\xymatrix{
(X_\xi, \bar{\iota}^*_{\xi, \circ} \scrA_X) \ar[r]^-{\bar{\iota}_\xi} \ar_{\bar{\pi}}[d] & (X, \scrA_X) \ar^{\pi}[d] \\
\Spec K \ar[r]_-{\iota_\xi} & S.
}
\end{gathered}
\end{align*}

\begin{definition} \label{dfn:GF}
We call
$(X_\xi, \bar{\iota}^*_{\xi, \circ} \scrA_X)$
the
\emph{generic fiber}
of
$\tilde{\pi}$.
\end{definition}

\begin{remark}
Another candidate for the generic fiber of
$\tilde{\pi}$
is the fiber product of
$(X, \scrA_X)$
and
$(\Spec K, \tilde{\iota}^*_{\xi, \circ} \scrA_S)$
over
$(S, \scrA_S)$,
where
$\tilde{\iota}_\xi
\colon
(\Spec K, \iota^*_\xi \scrA_S)
\to
(S, \scrA_S)$
is the strict morphism canonically induced by
$\iota_\xi$.
We adopt the above definition because of Corollary
\pref{cor:compatibility}
below,
i.e.,
compatibility with derived categories of twisted coherent sheaves.
\end{remark}

\subsection{Canonical functor from the Serre quotient}
Let
$\tilde{\pi}
\colon
(X, \scrA_X)
\to
(S, \scrA_S)$
be a flat proper morphism of Azumaya $\bfk$-varieties.
Then
$\pi_\circ \colon X \to S$
is flat proper
\cite[Lemma D.19]{Kuz06}.
Now,
assume further that
$\tilde{\pi}$
is smooth
and
$S$
is a regular affine $\bfk$-variety
$\Spec R$.
We denote by
$\coh(X, \scrA_X)_0 \subset \coh(X, \scrA_X)$
the Serre subcategory spanned by $R$-torsion sheaves,
i.e.,
for each
$\scrF \in \coh(X, \scrA_X)_0$
there is an element
$r \in R$
such that
$r \scrF = 0$.
This makes sense as
$\scrA_X$
is an $\scrO_X$-algebra.
We write
$\cC = \coh(X, \scrA_X) / \coh(X, \scrA_X)_0$
for the Serre quotient.
The natural projection
$p \colon \coh(X, \scrA_X) \to \cC$
which sends
$\scrF$
to
$\scrF_K$
is known to be exact.
By universality of Serre quotient,
the exact functor
\begin{align*}
(-) \otimes_R K
\colon
\coh(X, \scrA_X)
\to
\coh(X_\xi, \bar{\iota}^*_{\xi, \circ} \scrA_X)
\end{align*}
induces a unique exact functor
\begin{align*}
\Psi \colon \cC \to \coh(X_\xi, \bar{\iota}^*_{\xi, \circ} \scrA_X)
\end{align*}
such that
$(-) \otimes_R K = \Psi \circ p$. 

\begin{lemma}[{\cite[Proposition 2.3]{HMS11}}] \label{lem:Klin}
The abelian category
$\cC$
is $K$-linear
and
for any
$\scrF, \scrG \in \coh(X, \scrA_X)$
the natural projection
$p \colon \coh(X, \scrA_X) \to \cC$
induces a $K$-linear isomorphism
\begin{align} \label{eq:Klin}
\Hom_{\coh(X, \scrA_X)}(\scrF, \scrG) \otimes_R K
\cong
\Hom_{\cC}(\scrF_K, \scrG_K).
\end{align}
\end{lemma}
\begin{proof}
The proof of
\cite[Proposition 2.3]{HMS11}
works without modification.
\end{proof}

\begin{lemma}[cf. {\cite[Proposition 2.3]{Mora}}] \label{lem:GFeq}
The functor
$\Psi$
is fully faithful.
\end{lemma}
\begin{proof}
The images of
$\scrF_K, \scrG_K \in \cC$
under
$\Psi$
are respectively isomorphic to
\begin{align*}
\bar{\iota}^*_\xi \scrF
=
\bar{\iota}^*_{\xi, \circ} \scrF \otimes_{\bar{\iota}^*_{\xi, \circ} \scrA_X} \bar{\iota}^*_{\xi, \circ} \scrA_X
=
\bar{\iota}^*_{\xi, \circ} \scrF, \
\bar{\iota}^*_\xi \scrG
=
\bar{\iota}^*_{\xi, \circ} \scrG \otimes_{\bar{\iota}^*_{\xi, \circ} \scrA_X} \bar{\iota}^*_{\xi, \circ} \scrA_X
=
\bar{\iota}^*_{\xi, \circ} \scrG
\end{align*}
for some
$\scrF, \scrG \in \coh(X, \scrA_X)$,
as
$\bar{\iota}_\xi$
is strict.
We have
\begin{align*}
\Hom_{\coh(X_\xi, \bar{\iota}^*_{\xi, \circ} \scrA_X)} \left( \Psi( \scrF_K), \Psi (\scrG_K ) \right)
&=
\Hom_{\coh(X_\xi, \bar{\iota}^*_{\xi, \circ} \scrA_X)} \left( \bar{\iota}^*_{\xi, \circ} \scrF, \bar{\iota}^*_{\xi, \circ} \scrG \right) \\
&\cong
\Gamma \circ \bar{\iota}^*_{\xi, \circ} \underline{\Hom}_{\coh(X, \scrA_X)} \left( \scrF, \scrG \right) \\
&\cong
\iota^*_\xi \circ \pi_{*, \circ} \underline{\Hom}_{\coh(X, \scrA_X)} \left( \scrF, \scrG \right) \\
&\cong
\Hom_{\coh(X, \scrA_X)} \left( \scrF, \scrG \right) \otimes_R K \\
&\cong
\Hom_{\cC} \left( \scrF_K, \scrG_K \right),
\end{align*}
where
the first,
second,
third
and
fourth isomorphisms follow from
\cite[Lemma D.14]{Kuz06}, 
flat base change,
\cite[Lemma 2.4]{Mora}
and
\pref{lem:Klin}
respectively.
Note that
$\underline{\Hom}_{\coh(X, \scrA_X)}(\scrF, \scrG)$
may be regarded as an ordinary coherent sheaf on
$X$.
\end{proof}

\begin{theorem}[cf. {\cite[Theorem 2.5]{Mora}}] \label{thm:esurj}
The functor
$\Psi$
is a $K$-linear equivalence.
\end{theorem}
\begin{proof}
It remains to show the essential surjectivity of
$\Psi$.
We claim that
any
$\scrF_\xi \in \coh(X_\xi, \bar{\iota}^*_{\xi, \circ} \scrA_X)$
can be obtained as the cokernel of a morphism of locally free right $\bar{\iota}^*_{\xi, \circ} \scrA_X$-modules of finite rank.
Indeed,
since by
\cite[Lemma D.24, D.38]{Kuz06}
the fiber product
$(X \times_S \Spec K, \pr^*_{1, \circ} \scrA_X)$
is smooth proper over
$\Spec K$,
one can apply
\pref{lem:A}
to obtain a surjection
\begin{align*}
\bigoplus_i \scrO_{X_\xi}(-n_i D_i)^{\oplus r_i} \otimes_{\scrO_{X_\xi}} \bar{\iota}^*_{\xi, \circ} \scrE_X
\to
\scrF_\xi \otimes_{\bar{\iota}^*_{\xi, \circ} \scrA_X} \bar{\iota}^*_{\xi, \circ} \scrE_X
\end{align*}
in
$\coh(X_\xi, \bar{\iota}^*_{\xi, \circ} \alpha_X)$
from a locally free $\bar{\iota}^*_{\xi, \circ} \alpha_X$-twisted sheaf of finite rank.
Since by
\cite[Lemma D.15]{Kuz06}
we have
$\bar{\iota}^*_{\xi, \circ} \scrA_X
\cong
\underline{\End}(\bar{\iota}^*_{\xi, \circ} \scrE_X)
\cong
\bar{\iota}^*_{\xi, \circ} \scrE_X
\otimes_{\scrO_{X_\xi}}
\bar{\iota}^*_{\xi, \circ} \scrE^\vee_X$,
via
\pref{thm:1.3.7}
it induces a surjection
\begin{align*}
\bigoplus_i \scrO_{X_\xi}(-n_i D_i)^{\oplus r_i} \otimes_{\scrO_{X_\xi}} \bar{\iota}^*_{\xi, \circ} \scrA_X
\to
\scrF_\xi \otimes_{\bar{\iota}^*_{\xi, \circ} \scrA_X} \bar{\iota}^*_{\xi, \circ} \scrA_X
\cong
\scrF_\xi.
\end{align*}
Hence we may assume
$\scrF_\xi$
to be a locally free of finite rank,
as the essential image of
$\Psi$
is
a full abelian subcategory of
$\coh(X_\xi, \bar{\iota}^*_{\xi, \circ} \scrA_X)$
and
closed under taking cokernels.

Take an affine open cover
$\{ U_i \} ^m_{i=1}$
of
$X$
such that
$\scrA_X(U_i) \cong A_i$
and
the restriction of
$\scrF_\xi$
to each affine open subset
$V_i = U_i \times_R K$
of
$X_\xi$
is isomorphic to a free right
$\tilde{B}_i = \tilde{A}_i \otimes_R K$-module  
$F_i 
=
\tilde{B}^{\oplus N}_i$
with
$N < \infty$.
Let
$\phi_{ij}
=
\phi_i \circ \phi^{-1}_j
\colon
F_j |_{V_{ij}}
\to
F_i |_{V_{ij}}$
be isomorphisms on
$V_{ij} = V_i \cap V_j$
where
$\phi_i \colon F_\xi |_{V_i} \to F_i$
are trivializations with their inverses
$\phi^{-1}_i \colon F_i \to F_\xi |_{V_i}$.
From
$F_i$
we obtain a rank
$N$
free right $\tilde{A}_i$-module
$E_i = \tilde{A}^{\oplus N}_i$
with the same generators.
By construction tensoring
$K$
with
$E_i$
recovers
$F_i$.
Replacing
\cite[Lemma 2.1]{Mora}
with
\pref{lem:Klin},
by the same argument as in
\cite[Theorem 2.5]{Mora}
one can glue
$E_i$
to construct 
$\scrE \in \coh(X, \scrA_X)$
up to shrinking the base
$\Spec R$
such that
$\scrE \otimes_R K \cong \scrF_\xi$.
Let
$\tilde{\scrE} \in \coh(X_T, \scrA_X |_{X_T})$
be the gluing of
$E_i |_{U_{i, T}}$
for
$U_{i, T}
=
U_i \times_R T$
where
$\Spec T \subset \Spec R$
is the affine open subset from the proof of
\cite[Theorem 2.5]{Mora}.

Finally,
we extend
$\tilde{\scrE}$
to some
$\scrE \in \coh(X, \scrA_X)$.
By
\cite[Proposition 3.2]{Per08}
the twisted sheaf
$\tilde{\scrE} \otimes_{\scrA_X |_{X_T}} \scrE_X |_{X_T}$
extends to an object in
$\coh(X, \alpha_X)$.
Tensoring it with
$\scrE^\vee_X$
over
$\scrO_X$,
we obtain
$\scrE \in \coh(X, \scrA_X)$.
This gives an extension of
$\tilde{\scrE}$,
as we have
$\scrE_X |_{X_T} \otimes_{\scrO_{X_T}} \scrE^\vee_X |_{X_T}
\cong
\scrA_X |_{X_T}$.
Since the exact functor
$(-) \otimes_R K$
factorizes through
\begin{align*}
\coh(X, \scrA_X)
\to
\coh(X_T, \scrA_X |_{X_T})
\to
\coh(X_\xi, \bar{\iota}^*_{\xi, \circ} \scrA_X)
\end{align*}
and
it sends
$\scrE$
to
$\scrF_\xi$,
there is an object
$\scrE_K \in \cC$
which maps to
$\scrF_\xi$
under
$\Psi$.
\end{proof}

\subsection{Canonical functor from the Verdier quotient}
As the functor
$\Psi \colon \cC \to \coh(X_\xi, \bar{\iota}^*_{\xi, \circ} \scrA_X)$ 
is exact,
termwise application of
$\Psi$
defines a derived functor
$D^b(\Psi)
\colon
D^b (\cC)
\to
D^b(X_\xi, \bar{\iota}^*_{\xi, \circ} \scrA_X)$.
Via
\cite[Theorem 3.2]{Miy}
it coincides with a functor
\begin{align*}
\Phi
\colon
D^b(X, \scrA_X) / D^b_0(X, \scrA_X)
\to
D^b(X_\xi, \bar{\iota}^*_{\xi, \circ} \scrA_X),
\end{align*}
where
$D^b_0(X, \scrA_X)$
denotes the full triangulated subcategory spanned by complexes with $R$-torsion cohomology.
From
\pref{thm:esurj}
we obtain

\begin{corollary}[cf. {\cite[Corollary 2.6]{Mora}}] \label{cor:CGF}
The functor
$\Phi$
is an exact $K$-linear equivalence.
\end{corollary}

\subsection{The categorical generic fiber}
Now,
we drop the assumption imposed in the previous subsection
and
define the categorical generic fiber for a flat proper family of Azumaya varieties.

\begin{definition}
The
\emph{categorical generic fiber} \label{dfn:CGF}
of
$\tilde{\pi} \colon (X, \scrA_X) \to (S, \scrA_S)$
is the Verdier quotient
\begin{align*} 
D^b(X, \scrA_X) / \Ker(\bar{\iota}^*_\xi),
\end{align*}
where
$\Ker(\bar{\iota}^*_\xi)$
is the kernel of the exact $\bfk$-linear functor
$\bar{\iota}^*_\xi \colon D^b(X, \scrA_X) \to D^b(X_\xi, \bar{\iota}^*_{\xi, \circ} \scrA_X)$.
\end{definition}

\begin{remark}
Another candidate for the categorical generic fiber of
$\tilde{\pi}$
is the Verdier quotient
\begin{align*} 
D^b(X, \scrA_X) / \Ker(\tilde{\iota}^*_\xi).
\end{align*}
Recall that
$\tilde{\iota}_\xi
\colon
(\Spec K, \iota^*_{\xi, \circ} \scrA_S)
\to
(S, \scrA_S)$
is the strict morphism canonically induced by
$\iota_\xi$.
Also here,
we adopt the above definition because of Corollary
\pref{cor:compatibility}.
\end{remark}

\begin{lemma}
Let
$Z \subset X$
be a closed $\bfk$-subvariety with complement
$U$
and
$j_U \colon U \hookrightarrow X$
the open inclusion.
Then the equivalence
$\coh(X, \scrA_X) \simeq \coh(X, \alpha_X)$
from
\pref{thm:1.3.7}
induces an exact $\bfk$-linear equivalence
\begin{align*}
D^b(U, j^*_{U, \circ} \scrA_X)
\simeq
D^b(U, j^*_U \alpha_X).
\end{align*}
\end{lemma}
\begin{proof}
By
\cite[Proposition 4.3]{Per08}
the pullback functor
$j^*_U$
induces an equivalence
\begin{align} \label{eq:Perego}
\coh(X, \alpha_X) / \coh_Z(X, \alpha_X)
\simeq
\coh(U, j^*_U \alpha_X).
\end{align}
As any
$\scrF \in \coh(X, \scrA_X)$
is supported on
$Z$
if and only if
so is
$\scrF \otimes_{\scrA_X} \scrE_X$,
the equivalence
$\coh(X, \scrA_X) \simeq \coh(X, \alpha_X)$
induces
\begin{align*}
\coh(X, \scrA_X) / \coh_Z(X, \scrA_X)
\simeq
\coh(X, \alpha_X) / \coh_Z(X, \alpha_X).
\end{align*}
Now,
the claim follows from the equivalence
\begin{align} \label{eq:Perego-Azumaya}
\coh(U, j^*_{U, \circ} \scrA_X)
\simeq
\coh(X, \scrA_X) / \coh_Z(X, \scrA_X)
\end{align}
obtained in a similar way to
\pref{eq:Perego}.
The point is the fact that
any object in
$\coh(U, j^*_{U, \circ} \scrA_X)$
lifts to some object in
$\coh(X, \scrA_X)$,
as tacitly explained at the end of the proof of
\pref{thm:esurj}.
\end{proof}

For a flat proper family of Azumaya varieties,
the derived category of the generic fiber coincides with the categorical generic fiber in the following sense.

\begin{theorem}[cf. {\cite[Theorem 6.1]{Mora}}]
Let
$\tilde{\pi} \colon (X, \scrA_X) \to (S, \scrA_S)$
be a flat proper morphism of Azumaya $\bfk$-varieties.
Then there exists an exact $K$-linear equivalence
\begin{align*} 
D^b(X, \scrA_X) / \Ker(\bar{\iota}^*_\xi)
\simeq
D^b(X_\xi, \bar{\iota}^*_{\xi, \circ} \scrA_X).
\end{align*}
\end{theorem} 
\begin{proof}
Provided 
\pref{eq:Perego-Azumaya},
Corollary
\pref{cor:CGF}
and
\cite[Lemma D.24]{Kuz06},
the proof of
\cite[Theorem 6.1]{Mora}
carries over.
\end{proof}

\begin{remark}
The above theorem is a direct consequence of
\cite[Theorem 3.2]{Miy}
and
the $K$-linear equivalence
\begin{align*}
\coh(X) / \kker(\bar{\iota}^*_\xi)
\simeq
\coh(\bar{\iota}^*_{\xi, \circ} \scrA_X)
\end{align*}
which can be obtained similarly.
In particular,
\pref{thm:esurj}
also extends to nonaffine base case for flat proper morphisms of Azumaya $\bfk$-varieties.
\end{remark}

\begin{corollary} \label{cor:compatibility}
The equivalence
$\coh(X, \scrA_X) \simeq \coh(X, \alpha_X)$
from
\pref{thm:1.3.7}
induces an exact $K$-linear equivalence
\begin{align*}
D^b(X_\xi, \bar{\iota}^*_{\xi, \circ} \scrA_X)
\simeq
D^b(X, \alpha_X) / \Ker(\bar{\iota}^*_{\xi, \circ}).
\end{align*}
\end{corollary}
\begin{proof}
Since
$\bar{\iota}_{\xi}, \bar{\iota}_{\xi, \circ}$
are flat,
$\Ker(\bar{\iota}^*_\xi), \Ker(\bar{\iota}^*_{\xi, \circ})$
respectively coincide with the full triangulated subcategories
\begin{align*}
D^b_{\kker(\bar{\iota}^*_\xi)}(X, \scrA_X)
\subset
D^b(X, \scrA_X), \
D^b_{\kker(\bar{\iota}^*_{\xi, \circ})}(X, \alpha_X)
\subset
D^b(X, \alpha_X)
\end{align*}
spanned by complexes with cohomology in the Serre subcategories
\begin{align*}
\kker(\bar{\iota}^*_\xi)
\subset
\coh(X, \scrA_X), \
\kker(\bar{\iota}^*_{\xi, \circ})
\subset
\coh(X, \alpha_X).
\end{align*}
Here,
we use the symbol
$\kker$
to denote the kernel of exact functors of abelian categories.
Then by
\cite[Theorem 3.2]{Miy}
it suffices to show that
the equivalence sends
$\kker(\bar{\iota}^*_\xi)$
to
$\kker(\bar{\iota}^*_{\xi, \circ})$.
Any
$\scrF \in \coh(X, \scrA_X)$
belongs to
$\kker(\bar{\iota}^*_\xi)$
if and only if
there is an affine open subset
$\Spec R \subset S$
such that
$\scrF |_{\Spec R} \otimes_R K = 0$.
This is also a
necessary
and
sufficient
condition for its image
$\scrF \otimes_{\scrA_X} \scrE_X \in \coh(X, \alpha_X)$
under the equivalence to belong to
$\kker(\bar{\iota}^*_{\xi, \circ})$.
\end{proof}

\begin{remark}
As expected,
it immediately follows
\begin{align*}
D^b(X, \alpha_X) / \Ker(\bar{\iota}^*_{\xi, \circ})
\simeq
D^b(X_\xi, \bar{\iota}^*_{\xi, \circ} \alpha_X).
\end{align*}
For this reason,
we will also call the pair
$(X_\xi, \bar{\iota}^*_{\xi, \circ} \alpha_X)$
the
\emph{generic fiber}
of
$\tilde{\pi}$,
although it is not a locally ringed space.
\end{remark}

\subsection{Removal of torsion parts}
Let
$\tilde{\pi} \colon (X, \scrA_X) \to (S, \scrA_S), 
\tilde{\pi}^\prime \colon (Y, \scrA_Y) \to (S, \scrA_S)$
be flat proper morphisms of Azumaya $\bfk$-varieties.
Assume that
$\tilde{\pi}, \tilde{\pi}^\prime$
are smooth
and
$S$
is a regular affine $\bfk$-variety
$\Spec R$.
Assume further that
their generic fibers
$(X_\xi, \bar{\iota}^*_{\xi, \circ} \scrA_X),
(Y_\xi, \bar{\iota}^{\prime *}_{\xi, \circ} \scrA_Y)$
are derived-equivalent,
i.e.,
there is an exact $K$-linear equivalence
\begin{align*}
\Phi
\colon 
D^b(X_\xi, \bar{\iota}^*_{\xi, \circ} \scrA_X)
\to
D^b(Y_\xi, \bar{\iota}^{\prime *}_{\xi, \circ} \scrA_Y).
\end{align*}
We use the same symbol
$\Phi$
to denote the induced exact $K$-linear equivalence
\begin{align*}
\Phi
\colon 
D^b(X_\xi, \bar{\iota}^*_{\xi, \circ} \alpha_X)
\to
D^b(Y_\xi, \bar{\iota}^{\prime *}_{\xi, \circ} \alpha_Y).
\end{align*}
By
\pref{thm:main1}
there exists an object
\begin{align*}
P_K
\in
D^b(X_\xi \times_K Y_\xi, \bar{\iota}^*_{\xi, \circ} \alpha^{-1}_X \boxtimes \bar{\iota}^{\prime *}_{\xi, \circ} \alpha_Y).
\end{align*}
unique up to isomorphism
such that
the associated Fourier--Mukai transform
$\Phi_{P_K}$
is isomorphic to
$\Phi$.
Due to Corollary
\pref{cor:compatibility},
one can always lift it to some
$P
\in
D^b(X \times_R Y, \alpha^{-1}_X \boxtimes \alpha_Y)$
along the projection to the Verdier quotient.
From
\cite[Corollary 2.3.9, 2.4.3]{Cal}
it follows that
the relative integral functor
\begin{align*}
\Phi_P
\colon
D^b(X, \alpha_X)
\to
D^b(Y, \alpha_Y)
\end{align*}
admits
left
and
right
adjoints
$\Phi^L_P, \Phi^R_P$.
In the sequel,
we fix a strong generator
$E$
of
$D^b(X, \alpha_X)$,
which exists under our assumptions by 
\cite[Proposition 25]{Per}.

\begin{lemma}[cf. {\cite[Lemma 3.4.1]{BV}}] \label{lem:3.4.1}
The pullback
$E_K = \bar{\iota}^*_{\xi, \circ} E$
is a strong generator of
$D^b(X_\xi, \bar{\iota}^*_{\xi, \circ} \alpha_X)$.
\end{lemma}
\begin{proof}
The fact that
$E_K$
is compact follows from
\cite[Theorem 17, Proposition 25]{Per}.
For the rest,
the argument in
\cite[Lemma 2.5]{CS07}
works without modification.
\end{proof}

\begin{lemma}[cf. {\cite[Lemma 6.5]{Mora}}] \label{lem:counit2}
There exists a nonempty affine open subset
$U \subset \Spec R$
over which the restriction
\begin{align*}
\Phi_U
=
\Phi_{P_U}
\colon
D^b(X_U, \alpha_{X_U})
\to
D^b(Y_U, \alpha_{Y_U})
\end{align*}
induces bijections
\begin{align} \label{eq:counit}
\begin{gathered}
\Hom_{D^b(X_U, \alpha_{X_U})}(E_U, E_U)
\to
\Hom_{D^b(Y_U, \alpha_{Y_U})}(\Phi_U(E_U), \Phi_U(E_U)), \\
\Hom_{D^b(X_U, \alpha_{X_U})}(E_U, \Phi^L_U \Phi_U(E_U))
\to
\Hom_{D^b(Y_U, \alpha_{Y_U})}(\Phi_U(E_U), \Phi_U \Phi^L_U \Phi_U(E_U)).
\end{gathered}
\end{align}
\end{lemma}
\begin{proof}
By
\pref{lem:3.4.1}
the pullback
$E_K$
can not be trivial.
The same argument as in
\cite[Lemma 6.5]{Mora}
shows that
there exists a nonempty open subset
$U \subset \Spec R$
over which the restriction
$E_U$
is $\scrO_U(U)$-torsion free. 
Since we have
\begin{align*}
\Hom_{D^b(X_U, \alpha_{X_U})}(E_U, E_U)
\cong
\Hom_{X_U}(\scrO_{X_U}, E^\vee_U \otimes_{\scrO_{X_U}} E_U),
\end{align*}
by
\cite[Lemma 6.6]{Mora}
the $\scrO_U(U)$-module
$\Hom_{D^b(X_U, \alpha_{X_U})}(E_U, E_U)$
is coherent.
For the rest,
the argument in
\cite[Lemma 6.5]{Mora}
works without modification.
\end{proof}

Similarly,
one obtains the dual statement.

\begin{lemma}[cf. {\cite[Lemma 6.7]{Mora}}] \label{lem:unit2}
There exists a nonempty affine open subset
$U \subset \Spec R$
over which the restriction
\begin{align*}
\Phi^L_U
=
\Phi^L_{P_U}
\colon
D^b(Y_U, \alpha_{Y_U})
\to
D^b(X_U, \alpha_{X_U})
\end{align*}
induces bijections
\begin{align} \label{eq:unit}
\begin{gathered}
\Hom_{D^b(Y_U, \alpha_{Y_U})}(E^\prime_U, E^\prime_U)
\to
\Hom_{D^b(X_U, \alpha_{X_U})}(\Phi^L_U(E^\prime_U), \Phi^L_U(E^\prime_U)), \\
\Hom_{D^b(Y_U, \alpha_{Y_U})}(\Phi_U \Phi^L_U (E^\prime_U), E^\prime_U)
\to
\Hom_{D^b(X_U, \alpha_{X_U})}(\Phi^L_U \Phi_U \Phi^L_U (E^\prime_U), \Phi^L_U (E^\prime_U)).
\end{gathered}
\end{align}
\end{lemma}

\subsection{Specialization}
\begin{theorem}[cf. {\cite[Theorem 6.8]{Mora}}] \label{thm:main}
There exists a nonempty affine open subset
$U \subset \Spec R$
over which
$\Phi_U$
becomes an $\scrO_U(U)$-linear exact equivalence.
In particular,
over any closed point
$s \in U$
the closed fibers
$(X_s, \bar{\iota}^*_{s, \circ} \scrA_X),
(Y_s, \bar{\iota}^{\prime *}_{s, \circ} \scrA_Y))$
are derived-equivalent.
\end{theorem}
\begin{proof}
Provided
\pref{lem:3.4.1},
the bijections
\pref{eq:counit}
from
\pref{lem:counit2}
and
that
\pref{eq:unit}
from
\pref{lem:unit2},
the proof of
\cite[Theorem 6.8]{Mora}
carries over.
\end{proof}

\begin{corollary}[cf. {\cite[Corollary 6.9]{Mora}}] \label{cor:specialization}
Let
$\tilde{\pi} \colon (X, \scrA_X) \to (S, \scrA_S), 
\tilde{\pi}^\prime \colon (Y, \scrA_Y) \to (S, \scrA_S)$
be flat proper morphisms of Azumaya $\bfk$-varieties.
Assume that
their generic fibers are derived-equivalent.
Then there exists a nonempty open subset
$U \subset S$
to which the base changes
$(X_U, \bar{\iota}^*_{U, \circ} \scrA_X),
(Y_U, \bar{\iota}^{\prime *}_{U, \circ} \scrA_Y)$
become $U$-linear derived-equivalent.
In particular,
over any closed point
$s \in U$
the closed fibers
$(X_s, \bar{\iota}^*_{s, \circ} \scrA_X),
(Y_s, \bar{\iota}^{\prime *}_{s, \circ} \scrA_Y)$
are derived-equivalent.
\end{corollary}
\begin{proof}
The proof of
\cite[Corollary 6.9]{Mora}
carries over.
Note that
in its final step,
although it is not strictly necessary,
one has to adapt
\cite[Proposition 2.15]{HLS}
to twisted case in a straightforward way.
\end{proof}


\end{document}